\documentclass[10pt]{article}

\usepackage{anyfontsize, fix-cm}
\usepackage[english]{babel}
\usepackage{amssymb, amsthm, mathtools, microtype, xcolor} 
\usepackage{enumitem}

\usepackage[pagebackref, pdfencoding=auto,  psdextra]{hyperref}
\usepackage[abbrev, nobysame]{amsrefs} 

\renewcommand{\MR}[1]{} 
\renewcommand{\PrintDOI}[1]{}

\usepackage{tikz-cd}
\usetikzlibrary{babel, shapes.geometric, math, calc}

\newtheorem{theorem}{Theorem}[section]
\newtheorem{lemma}[theorem]{Lemma}

\newtheorem{corollary}[theorem]{Corollary}
\newtheorem{proposition}[theorem]{Proposition}

\newtheorem{Definition}[theorem]{Definition}
\newtheorem*{singerconjecture}{Singer Conjecture for $W_L$}

\theoremstyle{definition}
\newtheorem*{example}{Example}

\theoremstyle{remark}
\newtheorem{remark}[theorem]{Remark}

\newcommand{\zz}{\Z}

\newcommand{\rr}{\mathbb{R}} 
\newcommand{\Z}{\mathbb{Z}}
\newcommand{\Q}{\mathbb{Q}}

\DeclareMathOperator{\Sub}{Sub}
\DeclareMathOperator{\Lk}{Lk}
\DeclareMathOperator{\Cone}{Cone}
\DeclareMathOperator{\St}{St}

\newcommand{\cU}{\mathcal{U}}
\newcommand{\cu}{\mathcal{U}}

\newcommand{\gs}{\sigma} 
\newcommand{\gS}{\Sigma} 
\newcommand{\gt}{\tau}  
\newcommand{\gG}{\Gamma}

\newcommand{\F}{\mathbb{F}}
\newcommand{\Fp}{\F_{p}} 

\renewcommand{\d}{\partial} 

\DeclarePairedDelimiter\abs{\lvert}{\rvert}

 \tikzset{ 
 	every node/.style={font=\scriptsize},
	vertex/.style={color = black, draw, fill = black, inner sep = 0, minimum size=2pt, shape=circle }, 
	hexagon/.style={color = black, shape=regular polygon, regular polygon sides=6, minimum size=3.5cm, draw, outer sep=0 }, 
	every pin edge/.style={<-,shorten <=1pt, thin }, 
	deg5/.pic={
	\coordinate (v1) at (2,2);
	\coordinate (v2) at (-2,2);
	\coordinate (v3) at (-2,-2);
	\coordinate (v4) at (2,-2);
	\coordinate (y1) at (0,-2);
	\coordinate (y2) at (0,2);
	\coordinate (x1) at (-2,0);
	\coordinate (x2) at (2,0);
	\coordinate (a2) at (2,1);
	\coordinate (a1) at (-2,1);
	\coordinate (b1) at (-1,-2);
	\coordinate (b2) at (-1,2);
	\coordinate (c1) at (-2,-1);
	\coordinate (c2) at (2,-1);
	\coordinate (d2) at (1,2);
	\coordinate (d1) at (1,-2);
	\coordinate (e) at (1.5,1.5);
       \node[vertex, label=above right:$v$] at (v1) {}; 
       \node[vertex, label=above left:$v$] at (v2) {}; 
       \node[vertex, label=below left:$v$] at (v3) {}; 
       \node[vertex, label=below right:$v$] at (v4) {}; 
       \node[vertex, label=above:$y$] at (y2) {}; 
       \node[vertex, label=below:$y$] at (y1) {}; 
       \node[vertex, label=right:$x$] at (x2) {}; 
       \node[vertex, label=left:$x$] at (x1) {}; 
       
       \node[vertex, label=right:$a$] at (a2) {}; 
       \node[vertex, label=left:$a$] at (a1) {}; 
       \node[vertex, label=above:$b$] at (b2) {}; 
       \node[vertex, label=below:$b$] at (b1) {}; 
       \node[vertex, label=left:$c$] at (c1) {}; 
       \node[vertex, label=right:$c$] at (c2) {}; 
	\node[vertex, label=above:$d$] at (d2) {}; 
       \node[vertex, label=below:$d$] at (d1) {}; 
       
      \node[vertex, label=left:$e$] at (e) {};          
       \draw (v1)--(v2)--(v3)--(v4)--(v1);

        	\draw (a1)--(b2);
         \draw (b1)--(c1);
         \draw (c2)--(d1);
        	 \draw (d2)--(a2);
	 \draw (e)--(v1);
	 }
 }

\author{Grigori Avramidi
\and Boris Okun
\and Kevin Schreve}
\title{Reflections on branched covers}

\begin{document}
\maketitle
\begin{abstract}
    A flag complex $L$ determines a locally CAT(0) cube complex $P_L$ with the links of all vertices isomorphic to $L$.
    The fundamental group of $P_L$ is the commutator subgroup of the right-angled Coxeter group $W_L$.
    We observe that an edge contraction $L \rightarrow L/e$ determines a $2$-fold branched cover $P_L \rightarrow P_{L/e}$ (up to a homotopy equivalence).
    We use this branched cover to give an inequality between the $\F_2$-$L^2$-Betti numbers of $W_L$ and $W_{L/e}$.
    This reduces an $\F_2$-version of the Singer Conjecture for right-angled Coxeter groups to minimal flag triangulations of even dimensional spheres.
    We also use this inequality to compute the $L^2$-Betti numbers of $W_L$ when $L$ is a flag triangulation of the projective plane or the torus.
\end{abstract}

\section{Introduction} A classical result of Alexander is that two simplicial complexes are PL homeomorphic if and only if they can be connected by a sequence of stellar subdivisions and their inverses.
Lutz and Nevo showed this result still holds for the subclass of flag simplicial complexes: two flag complexes are PL homeomorphic if and only if they can be connected by a sequence of edge subdivisions and their inverses with the intermediate simplicial complexes all flag \cite{ln16}.

Associated to any flag simplicial complex $L$ there is a right-angled Coxeter group $W_L$ and a proper cocompact action of $W_L$ on a CAT(0) cube complex, the Davis complex $\Sigma_L$.
The commutator subgroup $C_L$ is torsion-free and finite index in $W_L$, and the quotient $P_L := \Sigma_L/C_L$ is thus a compact locally CAT(0) cube complex.
The link of each vertex in $\Sigma_L$ is isomorphic to $L$, so if $L$ is a flag triangulation of $S^{n-1}$ then $P_L$ is a closed aspherical $n$-manifold.

If $L/e$ is the result of contracting an edge $e \in L$ (this is more general than inverting an edge subdivision), then we observe that $P_L$ is homotopy equivalent to a $2$-fold branched cover of $P_{L/e}$.
In this paper, we use this branched cover to relate certain homological invariants of $P_L$ and $P_{L/e}$.

We now describe the invariants.
The commutator $C_L$ is residually torsion-free nilpotent, so the group ring $\F C_L$ embeds into its Linnell skew field $D_{\F C_L}$.
The local homology of $P_L$ with $D_{\F C_L}$-coefficients then gives associated skew field Betti numbers:
\[
b_{*}(P_{L}; D_{\F C_L}):=\dim_{D_{\F C_L}} H_{*}(P_L; D_{\F C_L}) = \dim_{D_{\F C_L}} H_{*}(D_{\F C_L} \otimes_{\F C_{L}}C_{*}(\gS_L)).
\]
If $\F = \Q$, then these Betti numbers are equal to the classical $L^2$-Betti numbers of $\Sigma_L$ with respect to the $C_L$-action.
If $\F = \F_p$, then these Betti numbers can be viewed as finite field analogues of the $L^2$-Betti numbers, and we shall denote $b^{(2)}_{*}(W_{L}; \F) :=\frac{1}{[W_L:C_L]} b_{*}(\gS_{L}; D_{\F C_L})$, as they are well-defined invariants of $W_L$.

In \cite{aos25}, we showed that subdividing edges of $L$ gives an inequality between $\F$-$L^2$-Betti numbers for $W_L$ for all $\F$.
In particular, if $L$ is a flag complex and an edge $e \in L$ has $b_{i-1}^{(2)}(W_{\Lk e}; \F) = 0$, then $ b_i^{(2)}(W_{\Sub_eL}; \F) \le b_i^{(2)}(W_{L}; \F) $.

If a flag complex $L/e$ is obtained by contracting an edge $e \in L$, then we get a similar inequality, but only for the $\F_2$-$L^2$-Betti numbers.
\begin{theorem}\label{t:contraction}
    If $e$ is an edge in a flag complex $L$ such that $L/e$ is flag and $b^{(2)}_{*}(W_{\Lk e}; \F_{2})=0$, then
    \[
    b^{(2)}_{*}(W_{L} ; \F_{2} ) \leq b^{(2)}_{*}(W_{L/e}; \F_{2}).
    \]
\end{theorem}

We do not know an analogue of Theorem \ref{t:contraction} for $\F$-$L^2$-Betti numbers over other fields (in particular $\F = \Q$).

The theorem has some applications to the study of (generalizations of) the Singer conjecture, in particular in low dimensions.
Since $P_{L}$ is a closed aspherical manifold when $L$ is a flag PL triangulation of a sphere, we have:
\begin{singerconjecture}
    If $L$ is a flag PL triangulation of $S^{n-1}$, then
    \[
    b^{(2)}_{i}(W_L ; \Q ) = 0 \text{ for } i \ne \frac{n}{2}.
    \]
\end{singerconjecture}

Davis and the second author conjectured that if $L$ is a triangulation of a closed genus $g$ surface, then $b^{(2)}_{2}(W_L; \Q) = g$ and all other $\Q$-$L^{2}$-Betti numbers vanish \cite{do01}*{Conjecture 11.5.1}.
(If the surface is $S^{2}$ this is just the $3$-dimensional case of the Singer conjecture.) The conjectures can also be stated over $\F_{2}$, and the universal coefficients imply that the $\F_{2}$ versions are stronger than the usual ones.

For surfaces, since the links of edges are two points, the corresponding Coxeter group is $\F$-$L^{2}$-acyclic, so by Theorem~\ref{t:contraction}, contracting edges raises $\F_2$-$L^2$-Betti numbers.
Therefore, it suffices to show the conjecture for triangulations where no edges can be contracted so that $L/e$ is flag.
We call these \emph{minimal} flag triangulations, and the condition is equivalent to every edge being contained in a full $4$-cycle in $L$.

It is not hard to see that the only minimal flag triangulation of $S^2$ is the octahedron $O^2$.
Since $P_{O^2}$ is a $3$-torus and hence satisfies the $\F_2$-Singer conjecture, Theorem~\ref{t:contraction} implies the $\F_2$-Singer conjecture, and therefore the usual Singer conjecture, for all $W_{S^2}$.
This gives a new, and from our perspective, more elementary proof of the Singer Conjecture for $W_{S^2}$.
Previous proofs were by combining either geometrization in the form of Andreev's theorem \cites{a70,a70a} with a computation by Dodziuk of the $L^2$-cohomology of $\mathbb{H}^3$ \cite{d79}, or Agol's fibering criterion \cite{a08} with L\"{u}ck's mapping torus theorem~\cite{l94a}.

For the higher genus surfaces, it is known that there are only finitely many minimal flag triangulations in each genus \cites{mm92, mn95}.
Fisk, Mohar, and Nedela showed that the real projective plane has only five minimal flag triangulations \cite{fmn94}, and it is easy to see that the conjecture holds for them.
This way we obtain the following:
\begin{theorem}\label{t:introRP2}
    If $L$ is a flag triangulation of $\rr\mathrm{P}^{2}$, then $b^{(2)}_{2}(W_L ; \F) = 1/2$ and $b^{(2)}_{\neq 2}(W_L ; \F) = 0$.
\end{theorem}

Mezler in his thesis \cite{m19} claims that an exhaustive computer search shows that there are $1196$ minimal flag triangulations of the torus.
For the higher genus, the exact numbers seem to be unknown.

The bulk of this paper is a combinatorial argument (avoiding case-by-case analysis of $1196$ triangulations!) showing the conjecture for the minimal triangulations of the torus.
We prove:
\begin{theorem}\label{t:itorii}
    If $L$ is a flag triangulation of $T^{2}$, then $b^{(2)}_{2}(W_L ; \F) = 1$ and $b^{(2)}_{\neq 2}(W_L ; \F) = 0$.
\end{theorem}

In the last section, we discuss an application of Theorem~\ref{t:contraction} to the $\F_{2}$-Singer conjecture in higher dimensions.
Note that we have previously shown in \cite{aos21}, \cite{aos25} that for odd $p$, the $\F_p$-version of the Singer conjecture is false for certain $W_L$, where $L$ is a flag triangulation of $S^6$.
However, it is still possible that it holds over $\F_{2}$.
Using inductive arguments, we obtain the following:
\begin{theorem}
    The $\F_2$-Singer conjecture holds for all PL flag triangulations of spheres if and only if it holds for all minimal PL flag triangulations of even-dimensional spheres.
\end{theorem}

\subsection*{Acknowledgements} We thank Karim Adiprasito for pointing out that the unsolvability of the sphere recognition problem implies the existence of infinitely many minimal flag triangulations of $S^n$,  $n \geq 5$.
This is explained in Remark \ref{r:1}.
The second author was partially supported by Simons Foundation grant MPS-TSM-00007773. 
The third author was partially supported by NSF grants DMS-2203325 and DMS-2505290.
No AI tool was used at any stage of this work.

\section{Edge contractions and branched covers}

Let $L$ be a flag simplicial complex.
Let $e = vw$ be an edge in $L$.
The \emph{edge subdivision} of $L$ along $e$, denoted $\Sub_e(L)$, is obtained by replacing the open star of $e$ with the cone on the suspended link of $e$.
The \emph{edge contraction} of $L$ along $e$, denoted $L/e$, is obtained by removing the interior of the union of stars $\St_L v \cup \St_L w$, and then coning off its boundary to a new vertex $q$.
If $L/e$ is itself a flag complex, then we say that $e$ is an \emph{admissible edge}.
If $L$ has no admissible edges, then we say $L$ is a \emph{minimal flag complex}.

Note that since $L$ is flag, the links of $v$ and $w$ intersect in the link of $e$, and hence the boundary of $\St_L v \cup \St_L w$ is $(\Lk_{L} v - w )\mathop{\cup}\limits_{\Lk_{L}e} (\Lk _{L}w - v)$.
Admissibility of $e$ is equivalent to the boundary being a full subcomplex of $L-e$, which in turn is equivalent to the boundary being full in $L$.

\subsection{Davis complexes}
\begin{Definition}
    A \emph{mirror structure} on a CW-complex $X$ consists of an index set $S$ and a family of subcomplexes $\{X_s\}_{s \in S}$.
    The subcomplexes $X_s$ are the \emph{mirrors} of $X$.
    We say $X$ is a mirrored space over $S$.
    Given $x \in X$, let $S(x) = \{s \in S \mid x \in X_s\}$.
\end{Definition}

Given a flag complex $L$ with vertex set $S$, let $K_L$ be the geometric realization of the poset of simplices of $L$.
Then $K_L$ is isomorphic to the cone on the barycentric subdivision of $L$, with the empty simplex corresponding to the cone point.
There is a canonical mirror structure on $K_L$ where the $s$-mirror $K_s$ is the geometric realization of the subposet of simplices containing the vertex $s$, which is isomorphic to the star of $s$ in the barycentric subdivision of $L$.
Let $W_L$ be the right-angled Coxeter group associated to $L$, and given a subset $T \subset S$, let $W_T$ be the right-angled Coxeter group with generators $t \in T$.
The \emph{Davis complex} $\Sigma_L$ is the basic construction
\[
\cu(W_L, K_L) := (W_L \times K_L) /\sim,
\]
where $(w,x) \sim (w',x')$ if and only if $x = x'$ and $wW_{S(x)} = w'W_{S(x)}$.

The Davis complex admits a natural cubical structure.
Firstly, $K_L$ can be naturally identified with a CAT(0) cubical subcomplex of $[0,1]^{S}$; it is precisely the union of subcubes of $[0,1]^{S}$ containing $(0,0, \dots, 0)$ corresponding to the collections of vertices which span simplices in $L$.
The link of the vertex $(0,0,\dots,0)$ is isomorphic to $L$, and the vertices of each cube can be identified with vertices of the barycentric subdivision of $L$.

For a vertex $s$, the mirror $K_s$ is the intersection of $K_L$ with the hyperplane $x_s = 1$.
It is naturally isomorphic to the cubical complex $K_{\Lk s}$.
This cubical structure naturally extends to $\Sigma_L$.

The commutator subgroup $C_L$ of $W_{L}$ has index $2^{\abs{S}}$ and acts freely on $\Sigma_L$; we will denote the quotient by $P_L$.
The abelianization $(\zz/2)^{S}$ acts by reflections on $P_L$ with strict fundamental domain $K_L$, in other words, $P_L$ can also be defined as the basic construction:
\[
P_L = \cu((\zz/2)^{S}, K_L),
\]
where the mirror structure on $K_L$ is the same as the above.

\subsection{Double branched covers} Now, suppose that $e=st$ is an admissible edge of $L$, $L/e$ the $e$-contraction of $L$, and $q$ the new vertex of $L/e$.
Define a mirror structure on $K_L$ with index set $S/e = S \cup \{q\} - \{s,t\}$ by letting $K_{q}$ be the union of $K_s$ and $K_t$, and keeping all other mirrors the same.
Denote the resulting mirrored complex by $K'_{L}$, see Figure~\ref{f:1}.
Let
\[
Q_{L/e} = \cu((\zz/2)^{S/e}, K'_L).
\]
Note that $Q_{L/e}$ is still a cube complex, but is no longer combinatorial.
In particular, cubes in $Q_{L/e}$ may meet along multiple faces.
\begin{figure}
    \centering
    \begin{tikzpicture}
        \begin{scope}[xshift = 3cm, yshift=-.82cm, scale=1.2]
            \draw[step=1cm] (0,0) grid (4,4);
            \foreach \x in {0,...,4} \foreach \y in {0,...,4}
            	\node[vertex] at (\x,\y) {};
            \foreach \x in {0.5,...,3.5} \foreach \y in {1,3}
            	\node[vertex] at (\x,\y) {};
            \node[above] at (.5, 1) {$K_s \cup K_t$};
            \node[below=0.16cm] at (2, 0) {$Q_{L/e} = \cu((\zz/2)^{S/e}, K'_L)$};
        \end{scope}
        \begin{scope}[yshift = 3cm]
            \node[regular polygon, regular polygon sides=5, minimum size=2 cm, draw, name=a] {};
            \foreach \i in {1,...,5} {
            	\draw (a.side \i) -- (a.center);
            	\node[vertex] at (a.corner \i) {}; 
	    }
            \node[above left] at (a.side 1) {$K_s$};
            \node[above right] at (a.side 5) {$K_t$};
            \node[below=0.2cm] at (a.side 3) {$K_L$};
        \end{scope}
        \begin{scope}
            \node[regular polygon, regular polygon sides=4, minimum size=2.3 cm, draw, name=a] {};
            \foreach \i in {side 2,side 3,side 4, 60, 120} 
            	\draw (a.\i) -- (a.center); 
            \foreach \i in {1,...,4} 
            	\node[vertex] at (a.corner \i) {}; 
            \node[vertex] at (a.side 1) {};
            \node[above] at (a.side 1) {$K_{s}\cup K_t$};
            \node[below=0.2cm] at (a.side 3) {$K'_L$};
        \end{scope}
    \end{tikzpicture}
    \caption{If $L$ is a $5$-cycle and we collapse an edge $e=st$, then $Q_{L/e}$ and $P_{L/e}$ are tori (with different cubical structures, and we have omitted the cubical structure on $Q_{L/e}$).
    The hyperbolic surface $P_L$ is a double branched cover of $Q_{L/e}$ at eight points, the $(\zz/2)^4$-orbit of $K_s \cap K_t$.}\label{f:1}
\end{figure}
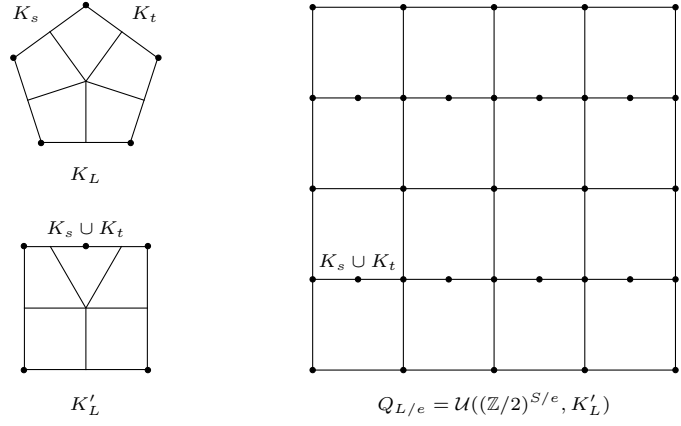
\begin{lemma}\label{l:branched1}
    $P_L$ is a double branched cover of $Q_{L/e}$, where the branching locus is a disjoint union of complexes each isomorphic to $P_{\Lk e}$.
\end{lemma}
\begin{proof}
    Let $\phi: (\zz/2)^{S} \rightarrow (\zz/2)^{S/e}$ send the generators $s$ and $t$ to the generator $q$.
    This induces a $(\zz/2)^{S}$ action on $Q_{L/e}$, and the inclusion $K_L' \rightarrow Q_{L/e}$ sending $x$ to $[1,x]$ has $i(K_s)$ contained in the fixed set of $s$ for each generator $s$.
    Therefore, the identity map $K_L \rightarrow K_L'$ induces a continuous map $f: P_L \rightarrow Q_{L/e}$ which sends $[w,x]$ to $[\phi(w), x]$.

    The preimage of a point $[w,x]$ in $Q_{L/e}$ are the points $[w,x]$ and $[wst,x]$ in $P_L$.
    These points are the same if and only if $st \in W_{S(x)}$, i.e.
    if $x$ is in the intersection of mirrors $K_s \cap K_t$.
    Therefore, $f$ is a homeomorphism restricted to the orbit of $K_s \cap K_t$ in $P_L$, which is a disjoint union of copies of $P_{\Lk e}$, and a $\Z/2$-cover outside the orbit.
\end{proof}
\begin{lemma}\label{l:branched2}
    $Q_{L/e}$ is homotopy equivalent to $P_{L/e}$.
\end{lemma}
\begin{proof}
    We will show that $Q_{L/e}$ is aspherical and $\pi_1(Q_{L/e}) = C_{L/e}$.
    The universal cover of $Q_{L/e}$ is the basic construction $\cU(W_{L/e}, K'_L)$, with the same mirror structure on $K'_L$.
    To show $\cU(W_{L/e}, K'_L)$ is contractible, since $L/e$ is flag it suffices to show that for each simplex $\gs \in L/e$, $K'_{\gs} = \bigcap_{s \in \gs} K'_s$ is contractible \cite{dbook}*{Theorem 9.1.4}.
    Note that an intersection of mirrors in the Davis chamber $K_{L}$ is either contractible or empty, depending on whether it corresponds to a simplex of $L$ or not.

    If $q \notin \gs$, then $K'_{\gs}=K_{\gs}$ is unchanged, and hence contractible.
    Therefore, suppose $q \in \gs$.
    Let $\gt = \gs - v_{e}$, we have
    \[
    K'_{\gs} = K'_{q} \cap K'_{\gt} = (K_s \mathop{\cup}\limits_{K_{st}} K_t) \cap K_{\gt} = K_{s \gt} \mathop{\cup}\limits_{K_{st\gt}} K_{t \gt}.
    \]
    Since
    \[
    \Lk_{L/e} q= (\Lk_{L}s - t) \mathop{\cup}\limits_{\Lk_{L}e} (\Lk_{L}t - s)
    \]
    and $\gt \in \Lk_{L/e} q$, if $\gt \notin \Lk e$, then exactly one of the terms in the above decomposition of $K'_{\gs}$ is empty and the other is contractible.
    If $\gt \in \Lk e$, then both terms and their intersection are contractible.
    Therefore, in both cases the union is contractible.

    The fundamental group of $Q_{L/e}$ can be computed just as the fundamental group of $P_{L/e}$, and is equal to the commutator subgroup $C_{L/e}$.
\end{proof}

\section{\texorpdfstring{$L^2$-Betti numbers of branched covers}{L2-Betti numbers of branched covers}}

Suppose $G$ is a residually torsion-free nilpotent group, and thus bi-orderable.
Therefore, $\F G$ embeds into a Malcev--Neumann skew field $D_<\F G$, and the division closure $D_{\F G}$ is independent of the choice of ordering.
If $X$ is a finite connected complex and $\pi_1X \rightarrow G$ is a homomorphism, then $D_{\F G}$ is a right $\pi_{1}X$-module and we can take homology of $X$ with local coefficients in $D_{\F G}$.
There are associated Betti numbers
\[
b_i(X; D_{\F G}):= \dim_{D_{\F G}}H_i(X; D_{\F G}) = \dim_{D_{\F G}}H_i( D_{\F G} \otimes_{\F \pi_1(X)} C_{*}(\widetilde X) ).
\]
If $X = \bigsqcup X_k$ is disconnected and each component has a map $\pi_{1} X_{k} \to G$, then we put
\[
b_i(X; D_{\F G}):= \sum b_i(X_{k}; D_{\F G}).
\]

Recall the following facts about these Betti numbers.
\begin{itemize}
    \item (Infimum formula) \cite{aos24}*{Theorem 3.6}
    \[
    b_i(X; D_{\F G}) = \inf_{X' \rightarrow X} \frac{b_i(X'; \F)}{\abs{X' \rightarrow X}},
    \]
    where $X'$ ranges over all finite index covers of $X$ pulled back from $BG$.

    \item (Weak universality) The infimum formula implies that if $H$ is also residually torsion-free nilpotent and $G \to H$ is a homomorphism, then
    \[
    b_i(X; D_{\F G}) \leq b_i(X; D_{\F H}).
    \]
    Moreover, if $G \to H$ is injective, then $D_{\F G} \subset D_{\F H}$, so
    \[
    D_{\F H} \otimes_{\F \pi_1(X)} C_{*}(\widetilde X) = D_{\F H} \otimes_{D_{\F G}} (D_{\F G} \otimes_{\F \pi_1(X)} C_{*}(\widetilde X)),
    \]
    and
    \[
    b_i(X; D_{\F G}) = b_i(X; D_{\F H}),
    \]
    since tensoring over skew fields is exact.

    \item (Universal coefficients) The infimum formula and the universal coefficient theorem imply that for any field $\F$,
    \[
    b_i(X; D_{\Q G}) \leq b_i(X; D_{\F G}).
    \]

    \item (Monotone in $p$-covers) \cite{blls14}*{Lemma 4.1 and Theorem 1.6} If $\gG$ is a $p$-group and $C_{*}$ is a chain complex of finitely generated free $\F_p\gG$-modules, then
    \[
    \dim_{\F_p}H_i(C_{*}) \le \abs{\gG}\dim_{\F_p} H_i(\F_p \otimes_{\F_{p}\gG} C_{*}).
    \]
    In particular, if $X'$ is a regular $p^r$-cover of $X$, then $b_i(X'; \F_p) \le p^rb_i(X; \F_p)$, and the infimum formula implies that $b_{*}(X'; D_{\F_{p} G}) \leq p^{r} b_{*}(X; D_{\F_{p} G})$.

    \item (Multiplicativity) \cite{aos24}*{Lemma 3.3} If $G'$ is a finite index subgroup of $G$ and $X'$ is the pullback of the cover $BG' \to BG$ to $X$, then
    \[
    b_i(X'; D_{\F G'}) =\abs{G:G'} b_i(X; D_{\F G}).
    \]
\end{itemize}
\begin{theorem}\label{t:branched}
    Suppose $X$ is a finite complex, and $X'$ is a degree $p^r$ regular branched cover of $X$.
    Suppose that the components $B_i$ of the branching locus $B$ are $\pi_1$-injective subcomplexes of $X$.
    Suppose that $X$ and $X'$ both have residually torsion-free nilpotent fundamental groups $G$ and $G'$ respectively, and $b_{*}(B_i; D_{\F_p \pi_{1}B_i}) = 0$.
    Then
    \[
    b_{*}(X'; D_{\F_p G'}) \leq p^{r} b_{*}(X; D_{\F_p G}).
    \]
\end{theorem}
\begin{proof}
    Denote $D = D_{\F_p G}$.
    As the components of $B$ are $\pi_1$-injective, by weak universality, $b_{*}(B; D) = \sum_i b_{*}(B_i; D_{\F_p \pi_{1}B_i}) = 0$.
    The long exact sequences of pairs $(X,B)$ and $(X',B)$ with coefficients in $D$
    \[
    \begin{tikzcd}[sep = small] \dots \arrow[r] &H_{*}(B; D) \arrow[r] & H_{*}(X; D) \arrow[r] &H_{*}(X, B; D) \arrow[r] &\dots\\
        \dots \arrow[r] &H_{*}(B; D) \arrow[r] & H_{*}(X'; D) \arrow[r] &H_{*}(X', B; D) \arrow[r] & \dots
    \end{tikzcd}
    \]
    show that $b_{*}(X; D) = b_{*}(X, B; D)$ and $b_{*}(X'; D) = b_{*}(X', B; D)$.
    The pair $(X', B)$ is a regular $p^{r}$-cover of $(X, B)$, so monotonicity implies that $b_{*}(X'; D) \leq p^{r} b_{*}(X; D)$.
    By weak universality, $b_{*}(X'; D_{\F_p G'}) \leq b_{*}(X'; D)$, hence the claim.
\end{proof}

Given a flag complex $L$, the commutator subgroup $C_L$ of $W_L$ is residually torsion-free nilpotent, and $[W_L:C_L] = 2^{|S|}$.
We define $b^{(2)}_{*}(W_{L}; \F) :=\frac{1}{2^{|S|}} b_{*}(P_{L}; D_{\F C_L})$.
\begin{theorem}\label{t:inequality}
    If an admissible edge $e$ in a flag complex $L$ has $b^{(2)}_{*}(W_{\Lk e}; \F_{2})=0$, then
    \[
    b^{(2)}_{*}(W_{L} ; \F_{2}) \leq b^{(2)}_{*}(W_{L/e}; \F_{2}).
    \]
\end{theorem}
\begin{proof}
    Combining Lemma \ref{l:branched1}, Lemma \ref{l:branched2}, and Theorem \ref{t:branched}, we get
    \[
    b^{(2)}_{*}(W_{L} ; \F_{2}) = \frac{1}{2^{\abs{S}}} b_{*}(P_{L}; D_{\F_2C_{L}}) \le \frac{2}{2^{\abs{S}}} b_{*}(P_{L/e}; D_{\F_2C_{L/e}}) = b^{(2)}_{*}(W_{L/e} ; \F_{2}),
    \]
    since $\abs{S/e} = \abs{S}-1$.
\end{proof}

\section{Minimal flag triangulations of surfaces}

Let $L$ be a simplicial complex.
An \emph{$n$-cycle} in $L$ is a subcomplex with $n$ vertices homeomorphic to a circle.
A cycle is \emph{trivial} if it is nullhomotopic in $L$, and \emph{nontrivial} otherwise.
Note that if $L$ is flag, then a nontrivial $4$-cycle is automatically full.
An \emph{annulus} in $L$ is a subcomplex homeomorphic to a planar region between two disjoint nested circles.
A \emph{pinched annulus} in $L$ is a subcomplex homeomorphic to a planar region between two tangent nested circles.
An \emph{edge-pinched annulus} in $L$ is a subcomplex homeomorphic to a planar region between two nested circles sharing an arc.
An \emph{annulus with a simplex} in $L$ is a subcomplex which is the union of an annulus with one extra simplex, obtained by choosing a vertex in one boundary component and joining it with a simplex in the other component.
A \emph{$4$-annulus} (pinched, edge-pinched, with a simplex) is a homotopically nontrivial annulus (pinched, edge-pinched, with a simplex, respectively) with $4$-cycle boundary components.

We will often use the following well-known characterization of admissible edges.
\begin{lemma}
    An edge $e$ in a flag complex $L$ is admissible if and only if it is not contained in a full $4$-cycle.
\end{lemma}
\begin{proof}
    If $e$ is in a full $4$-cycle, then contracting $e$ results in a full $3$-cycle, hence $L/e$ is not flag.
    On the other hand, if $e$ is not in an induced $4$-cycle then $(\Lk_{L} v - w )\mathop{\cup}\limits_{\Lk_{L}e} (\Lk _{L}w - v)$ is full in $L-v-w$, so $L/e$ is flag.
\end{proof}

Suppose now $L$ is a minimal flag triangulation of a surface and $v$ is a degree $4$ vertex, $\Lk v= abcd$.
By minimality, the edge $av$ is contained in a full $4$-cycle of the form $avcx$ with $x \notin \St v$.
Similarly, we have another full $4$-cycle $bvdy$.
We will refer to these $4$-cycles as \emph{diagonal full $4$-cycles}.
\begin{lemma}\label{l:octahedron}
    Suppose $L$ is a minimal flag triangulation of a surface and $v$ is a degree $4$ vertex.
    If one of the diagonal full $4$-cycles is trivial, then $L$ is an octahedron.
\end{lemma}
\begin{proof}
    The intersection number between the diagonal full $4$-cycles is 0, because one of them is trivial.
    Since the cycles cross each other at $v$, they must have another point of intersection.
    Thus, $x = y$ and, since $L$ is a flag triangulation of a surface, $L$ is the octahedron with vertices $v, a, b, c, d, x$.
\end{proof}
\begin{lemma}\label{l:4-cycle}
    Suppose $L$ is a minimal flag triangulation of a surface and $C$ is a trivial full $4$-cycle in $L$.
    Then $C$ is the link of a vertex of degree $4$.
\end{lemma}
\begin{proof}
    Since $C$ is nullhomotopic, it bounds a disk $D$, which is necessarily full.
    Since $C$ is full, if $D$ has more than one interior vertex, $D$ has a pair of adjacent interior vertices.
    By minimality, the resulting edge is contained in another full $4$-cycle $C'$, which has to be contained in $D$ and therefore is also nullhomotopic.
    Repeating this argument, we obtain a nested sequence of disks, which has to end with the star of an interior vertex of $D$ of degree $4$.
    Since the boundary of each subsequent disk contains an interior edge of the previous one, two diagonal full $4$-cycles passing through this vertex also contain interior edges of the previous disk.
    Then it follows that the diagonal cycles are contained in $D$ and therefore trivial, and hence $D$ contains an octahedron, which is a contradiction.
    So $D$ has a single interior vertex, which proves the claim.
\end{proof}
\begin{corollary}\label{c:edge}
    Suppose $L$ is a minimal flag triangulation of a surface which is not an octahedron.
    Then every edge of $L$ is contained in a nontrivial $4$-cycle.
\end{corollary}
\begin{proof}
    Let $e$ be an edge of $L$ and $C$ be a full $4$-cycle containing $e$.
    If $C$ is nontrivial, then we are done, otherwise by Lemma~\ref{l:4-cycle} $C$ is a link of a vertex.
    Then pushing the interior part of a diagonal cycle through this vertex to the link gives a $4$-cycle which contains $e$ and is nontrivial, since $L$ is not an octahedron.
\end{proof}

The following result is well-known, see for example \cite{v15}.
It follows immediately from Corollary~\ref{c:edge}, since all cycles in $S^{2}$ are trivial.
\begin{corollary}\label{c:sphere}
    The only minimal flag triangulation of $S^2$ is the octahedron.
\end{corollary}
\begin{lemma}\label{l:full cycle}
    Suppose $L$ is a minimal flag triangulation of a surface and $D$ a full subcomplex of $L$ which is homeomorphic to a $2$-disk.
    Then $D - \d D$ is a disjoint union of vertices, and in particular, if $\d D$ is also full, then $D = \Cone(\d D)$.
\end{lemma}
\begin{proof}
    Any full $4$-cycle that has an edge in $D - \d D$ is contained in $D$, and hence is trivial.
    Therefore, by Lemma~\ref{l:4-cycle}, it is a link of a vertex in $D - \d D$.

    If there exists an edge in $D - \d D$, then it is contained in a full $4$-cycle, which by the above is a link of some vertex in $D - \d D$.
    Then, again by the above, two diagonal full $4$-cycles passing through this vertex are trivial, and hence produce an octahedron in $L$.
\end{proof}
\begin{lemma}\label{l:vertex}
    Suppose $L$ is a minimal flag triangulation of a surface and not an octahedron.
    Then for any vertex, there exist two nontrivial $4$-cycles which cross exactly in that vertex.
\end{lemma}
\begin{proof}
    Let $v \in L$ be a vertex.
    By Corollary~\ref{c:edge}, every edge starting at $v$ is contained in a nontrivial $4$-cycle.
    Consider a nontrivial $4$-cycle $C$ through $v$ which makes the sharpest turn at $v$, i.e.
    which minimizes the distance in $\Lk v$ between two vertices where it intersects $\Lk v$ among all such $4$-cycles.
    Denote the vertices $C \cap \Lk v$ by $u$ and $u'$, they separate $\Lk v$ into two arcs, $I$ and $J$, where $I$ realizes the distance between them.
    The length of $I$ is at least $2$, since $C$ is full.
    Let $w$ be a vertex in $I$ adjacent to $u$.
    Then a nontrivial $4$-cycle $C'$ containing the edge $vw$ also intersects the interior of $J$, since otherwise it would make a sharper turn.
    So $C$ and $C'$ cross each other at $v$.

    They do not have any other vertex in common, since if $v'$ were such a vertex, then the $4$-cycle $vu'v'w$ would be homologous to $C$ and making a sharper turn than $C$.
    So $C$ and $C'$ cross exactly at $v$.
\end{proof}
\begin{lemma}\label{l:s2+e}
    Suppose $L$ is a relatively minimal flag triangulation of $S^{2}$ with a simplex $\gs=u*\gt$ with $u$ a vertex (the only admissible edges are in $\gs$).
    Then the link in $S^{2}$ of $u$ is a $4$-cycle.
\end{lemma}
\begin{proof}
    Let $e$ be an edge in $\Lk_{S^{2}} u$, and let $C_{e}$ be a full $4$-cycle through $e$.
    Then $C_{e}$ avoids $u$ and $\gt$, so $C_{e} \subset S^{2}$ and hence bounds two disks.
    If one of the disks contains neither $u$ nor $\gt$, then, arguing as in Lemma \ref{l:4-cycle}, we get that the disk is the star of a degree $4$ vertex.
    The two diagonal $4$-cycles through this vertex cannot have edges in $\gs$, so they produce an octahedron, and hence $S^{2}$ is an octahedron and we are done.

    If $C_{e}$ separates $u$ and $\gt$ in $S^{2}$, then it bounds a disk $D$ with $u \in D$ and $\gt \not\subset D$.
    Then $\Lk_{S^{2}} u \subset D$, and if $\Lk_{S^{2}} u$ is a $> 4$-cycle, it must have an edge $e'$ in the interior of $D$.
    Since $C_e$ is full, $C_{e'} \subset D$, and also avoids $u$ and $\gt$ and produces two disks.
    As before for one of the disks $D'$, we have $u \in D'$ and $\gt \not\subset D'$.
    Then $D' \subset D$, and we get a nested sequence of disks which never stops.
    Hence $\Lk_{S^{2}} u$ is a $4$-cycle.
\end{proof}

\section{\texorpdfstring{Triangulations of $T^{2}$ and $\rr \mathrm{P}^{2}$}{Triangulations of T2 and RP2}}

The main result of this section are Theorems~\ref{t:torus} and \ref{t:RP2}, which compute the $\F_{2}$-$L^{2}$-Betti numbers of RACGs based on flag triangulations of the torus and the real projective plane, respectively.

The initial step of our proof is a reduction to minimal flag triangulations using Theorem~\ref{t:inequality}.
After that, we use two methods to show $b^{(2)}_{1}(W_L) = 0$.

The first is to find three disjoint homotopically nontrivial $4$-cycles in $L$, in which case a spectral sequence argument (Lemma~\ref{l:ss}) proves the result.
For example, we find such cycles in 6-regular minimal flag triangulations of the torus in Lemma~\ref{l:deg6}.

The second method is to find a `seed' subcomplex $A$ with $b^{(2)}_{1}(W_A) = 0$, so that $L$ is obtained from $A$ by adding vertices whose links intersect the previous subcomplex in non-simplices.
This method suffices for $\rr \mathrm{P}^{2}$.

In general, we use the following definition.
\begin{Definition}
    A full subcomplex $M$ of a flag complex $L$ is \emph{sated} if for every vertex $s \in L - M$, $\Lk s \cap M$ is a (possibly empty) simplex.
\end{Definition}
For a full subcomplex $A \subset L$, we define $M(A)$ to be the intersection of all sated subcomplexes containing $A$.
(Note that the intersection of two sated subcomplexes is sated.) In other words, $M(A)$ is the maximal full subcomplex obtained from $A$ by adding vertices of $L$ whose links intersect the previous subcomplex in non-simplices.

If a seed subcomplex $A$ has $b^{(2)}_{1}(W_A) = 0$, then by an easy Mayer--Vietoris argument, $b^{(2)}_{1}(W_{M(A)}) = 0$.
We list several possible such seeds in Corollary~\ref{c:annulussimplex}.

The proof for the torus goes by finding a suitable seed $A$, and then finding $3$ disjoint $4$-cycles under the assumption $M(A) \ne L$.
Our main technical tool, Lemma~\ref{l:M<L}, produces a nontrivial $4$-cycle, which in subsequent lemmas we use either to find three disjoint $4$-cycles or to improve the seed.
Finally, in Corollary~\ref{c:deg4} and Lemma~\ref{l:deg5}, we analyze the degree $4$ or $5$ vertices to get a suitable seed, which completes the proof.

\subsection{Sated subcomplexes}

We now collect some properties of certain sated subcomplexes.
The first lemma follows straight from the definition.
\begin{lemma}\label{l:not2}
    Suppose $L$ is a flag triangulation of a surface and $M \subset L$ is a sated subcomplex.
    If an edge $e$ in $L$ has an endpoint not in $M$, then one of the two vertices in $\Lk e$ is not in $M$.
\end{lemma}
\begin{lemma}\label{l:notinM}
    Suppose $L$ is a flag triangulation of a surface and $M \subset L$ is a sated subcomplex.
    Suppose $qr$ is an edge in $\Lk m - M$, where $m \in M$, and suppose $qrst$ is a nontrivial $4$-cycle.
    Then at least one of $s$, $t$ is not in $M$.
\end{lemma}
\begin{proof}
    Indeed, if $s$ is in $M$, then it must be adjacent to $m$, since $s$ and $m$ are in the link of $r$, and $r$ is not in $M$.
    Similarly for $t$.
    So, if both are in $M$, the cycle $qrst$ is the link of $m$, hence trivial, a contradiction.
\end{proof}
\begin{lemma}\label{l:M=L}
    Suppose $L$ is a minimal flag triangulation of a torus, $M \subset L$ is a sated subcomplex, and $C \subset M$ is a nontrivial $4$-cycle.
    If $M$ contains the links of all vertices of $C$, then $M = L$.
\end{lemma}
\begin{proof}
    Since $L$ is a torus, for any vertex $v$, the two nontrivial $4$-cycles through $v$ provided by Lemma~\ref{l:vertex} are homologically different, and hence one of them must cross $C$.
    Therefore, by the assumption on links, it has at least three vertices in $M$.
    Then $v$ is also in $M$.
\end{proof}
\begin{lemma}\label{l:M<L}
    Suppose $L$ is a minimal flag triangulation of a torus, $M \subsetneq L$ is a proper sated subcomplex, and $C \subset M$ is a nontrivial $4$-cycle.
    Then there exists a vertex $c \in C$ and a nontrivial $4$-cycle $qrst$, such that $qr \subset \Lk c$ and $q, r, s \notin M$.
    Moreover, if $t \in M$, then $t \in \Lk c$ and $s \notin \Lk c$.
\end{lemma}
\begin{proof}
    By Lemma~\ref{l:M=L}, there exist $c \in C$ and $q \in \Lk c - M$.
    By Lemma~\ref{l:not2}, there exists $r \in \Lk(qc) - M$.
    Consider a nontrivial $4$-cycle $qrst$ through the edge $qr$ provided by Corollary~\ref{c:edge}.
    By Lemma~\ref{l:notinM}, one of $s$, $t$ is not in $M$.
    So, after possibly renaming $q \leftrightarrow r$ and $s \leftrightarrow t$, we can assume that $s \notin M$.
    To finish the proof, if $t \in M$, then $c, t \in \Lk q$ and $q \notin M$ imply that $t \in \Lk c$, and therefore $s \notin \Lk c$ because $qrst$ is nontrivial.
\end{proof}
\begin{lemma}\label{l:acyclicsubcomplex}
    Suppose $L$ is a minimal flag triangulation of a torus, and $M \subsetneq L$ is a proper sated subcomplex, containing two disjoint nontrivial $4$-cycles $C_1$ and $C_2$.
    Then $L$ contains either three disjoint nontrivial $4$-cycles or a full pinched $4$-annulus.
\end{lemma}
\begin{proof}
    By Lemma~\ref{l:M<L}, we have a nontrivial $4$-cycle $qrst$ with $q, r, s \notin M$.
    If $t \notin C_{1} \cup C_{2}$, then the three cycles $qrst$, $C_{1}$, and $C_{2}$ are disjoint.
    Otherwise, we can assume $t \in C_{1}$.
    Then we get a full pinched $4$-annulus between $C_{1}$ and $qrst$ in $L-C_{2}$.
\end{proof}
\begin{lemma}\label{l:crossing cycles}
    Suppose $L$ is a minimal flag triangulation of a torus, and $M \subsetneq L$ is a proper sated subcomplex, containing an interior vertex $v$: $\St v \subset M$.
    Then $L$ contains two disjoint nontrivial $4$-cycles.
    Moreover, if there exist two $4$-cycles $C_{1}$ and $C_{2}$ crossing at $v$ such that for all vertices $u_{1} \in C_{1} - v$ and $u_{2} \in C_{2} - v$ the intersections $\Lk u_{1} \cap C_{2}$ and $\Lk u_{2} \cap C_{1}$ are connected, then $L$ contains a full $4$-annulus with a simplex.
\end{lemma}
\begin{proof}
    By Lemma~\ref{l:vertex} there exist two $4$-cycles $C_{1}$ and $C_{2}$ crossing at $v$.
    Note that $C_{1} \cup C_{2} \subset M$.
    Then by Lemma~\ref{l:M<L}, taking $C_{1}$ as a $4$-cycle, we have a nontrivial $4$-cycle $qrst$ with $q, r, s \notin M$ and hence this cycle cannot go through $v$.
    Since it is nontrivial, it intersects one of $C_{1}$ or $C_{2}$ at $t$, and is disjoint from the other.
    This proves the first claim.

    To prove the second, we will assume $t \in C_{1}-v$, and then we have two annuli between $qrst$ and $C_{2}$.
    Suppose we have an edge from $t$ to $C_{2}$.
    Then we claim that the $4$-annulus $A$ between $C_{2}$ and $qrst$ which contains this edge is full with a simplex.
    Note that all such edges from $t$ are in $A$, since $\Lk t \cap C_{2}$ is connected by hypothesis.
    Since $q, s \notin M$ are adjacent to $t \in M$, edges from them to $C_{2}$ must end in $\Lk t \cap C_{2}$, and therefore must be in $A$.
    So the only possible edges between $qrst$ and $C_{2}$ outside of $A$ are from $r$, and they must end in adjacent points.

    If there is no edge from $t$ to $C_{2}$, then there are no edges from $q$ or $s$ to $C_{2}$, and one of the annuli is full and the other is full with a simplex.
\end{proof}

\subsection{\texorpdfstring{$4$-annuli}{4-annuli}}

A full $4$-annulus $A$ is \emph{maximal} if there is no full $4$-annulus $A'$ with $A \subsetneq A'$.
\begin{lemma}\label{l:full annulus}
    Suppose $L$ is a minimal flag triangulation of a torus and $A \subset L$ is a maximal full $4$-annulus.
    If $M(A) \neq L$, then $L$ contains three disjoint nontrivial $4$-cycles.
\end{lemma}
\begin{proof}
    By Lemma~\ref{l:M<L}, taking one of the boundary components as $C$, we have a nontrivial $4$-cycle $qrst$ and a vertex $c \in C$ with $q, r \in \Lk c$, and $q, r, s \notin M(A)$.
    We claim that $t \notin \d A$, hence $\d A$ and $qrst$ are three disjoint nontrivial $4$-cycles.

    Suppose $t \in \d A$.
    Then by Lemma~\ref{l:M<L}, $t \in \Lk c$.
    Note that the link of a vertex $\notin M(A)$ cannot contain vertices in different boundary components, since an edge between those vertices would either contradict the fullness of $A$ or create an empty triangle in $L$.

    This implies that $t \in C$, since $r \notin M(A)$ is adjacent to $c$ and $t$.
    It also implies that $q$, $r$, and $s$ are not adjacent to the vertices in the other boundary component, since $q, r, s \notin M(A)$, and $q$ and $r$ are adjacent to $c$, and $s$ is adjacent to $t$.
    Hence the annulus bounded by $qrst$ and the other boundary component which contains $A$ is full, contradicting the maximality of $A$.
\end{proof}
\begin{lemma}\label{l:non-full}
    Suppose $A \subset L$ is a full pinched $4$-annulus, and let $p \in \d A$ be the pinched vertex.
    Suppose $C \subset L - A$ is a nontrivial $4$-cycle, and let $A'$ and $A''$ be the two annuli bounded by $C$ and $4$-cycles of $\d A$ which do not contain the other $4$-cycle of $\d A$.
    If both $A'$ and $A''$ are non-full, then $C$ intersects $\Lk p$ in exactly two non-adjacent points.
\end{lemma}
\begin{proof}
    If $A'$ is not full, then there is an edge between its boundary components outside of $A'$.
    This edge must intersect the other $4$-cycle of $\d A$, and therefore must connect $p$ and $C$.
    Similarly, there is an edge for $A''$.
    Since these edges are in different annuli, if their endpoints in $C$ were adjacent, or if there was another edge from $p$ to $C$, we would get an empty triangle.
    Hence the claim.
\end{proof}
\begin{lemma}\label{l:pinchedannulus}
    Suppose $L$ is a minimal flag triangulation of a torus and $A \subset L$ is a full pinched $4$-annulus.
    If $M(A) \neq L$, then $L$ contains a full $4$-annulus.
\end{lemma}
\begin{proof}
    Let $p \in \d A$ denote the pinched vertex.
    Note that if $v \notin M(A)$, then $\Lk v$ cannot contain vertices in different components of $\d A - p$, since those vertices cannot be adjacent in $L$, as that would create an empty triangle in $L$.

    Then by Lemma~\ref{l:M<L}, taking one of the boundary components as a $4$-cycle, we have a nontrivial $4$-cycle $qrst$ and a vertex $c \in \d A$ with $q, r \in \Lk a$ and $q, r, s \notin M(A)$.

    We first consider the case $c=p$.
    Then, since $q, r \in \Lk p$ are adjacent, if $t \notin \d A$ we have a full $4$-annulus by Lemma~\ref{l:non-full}.

    So we can assume $t \in \d_{1} A$ and therefore $t \in \Lk p$ by Lemma~\ref{l:M<L}.
    Then the cycles $qrst$ and $\d_{2}A$ are disjoint and we get two annuli between them.
    Since $\d_{1} A$ has single point intersections with $qrst$ and $\d_{2}A$ and is homotopic to them, it stays in one of the annuli, which we will denote by $A'$.
    Then $A \subset A'$, and there are no edges from $t$ to $\d_{2} A -p$ outside of $A'$, since $A$ was full.
    Since $pt$ is an edge of $\d_{1} A$, the triangles $prt$ and $pqr$ are inside $A'$.
    Therefore the edges $pq$ and $pr$ are inside $A'$, and since $p$ is not adjacent to $s$, there are no edges from $p$ to the other boundary component outside of $A'$.
    Since $r, s \notin M(A)$ are adjacent to $t \in \d_{1}A$, there are no edges from them to $\d_{2} A -p$ outside of the annulus.
    Since $q \notin M(A)$ is adjacent to $p \in \d_{1}A$, an edge from $q$ to $\d_{2} A -p$ must end in $\Lk p \cap \d_{2}A$, which is impossible, since $q$ and $\Lk p \cap \d_{2}A$ are separated by $\Lk p \cap \d_{1}A$ in $\Lk p$.
    So the annulus $A'$ is full.

    Thus we can assume that $c \neq p$ and $\Lk p \subset M(A)$.
    Then $q, r, s \notin \Lk p$ and by Lemma~\ref{l:non-full}, $t \in \d A$ and hence $c$ and $t$ are adjacent, and we can assume $c, t \in \d_{1} A$.
    Note that $t \neq p$, since $\Lk p \subset M(A)$.
    Therefore the cycles $qrst$ and $\d_{2}A$ are disjoint and we again get an annulus $A' \supset A$ between them.
    Now $q, r, s$ cannot be adjacent to $p$ since $\Lk p \subset M(A)$.
    They also cannot be adjacent to vertices of $\d_{2}A -p$ because they are adjacent to vertices in $\d_{1}A -p$.
    As before, the edges from $t$ to $\d_{2}A$ are in $A$, hence in $A'$.
    So $A'$ is full.
    This finishes the proof.
\end{proof}

\subsection{\texorpdfstring{$6$-regular tori}{6-regular tori}}

Let $L$ be a $6$-regular flag triangulation of $T^2$, i.e.
all the vertices of $L$ have degree $6$.
Note that the universal cover of any $6$-regular $L$ is the standard 6-regular triangulation of $\rr^{2}$, made up of vertical, horizontal, and diagonal lines.
This identifies $\pi_{1}L$ with a finite index subgroup of $\zz^{2}$, and the quotient group $G$ acts on $L$ simplicially and simply transitively on vertices.
By choosing a base vertex, we identify vertices of $L$ with $G$.
Any flag triangulation of the torus has at least $12$ vertices \cite{hr91}*{Theorem 1}, \cite{bowwzz20}*{Theorem 4.8}, and hence $\abs{G} \geq 12$.
\begin{lemma}\label{l:deg6}
    If $L$ is a $6$-regular minimal flag triangulation of a torus, then $L$ contains three disjoint nontrivial $4$-cycles.
\end{lemma}
\begin{proof}
    The $\zz^2$-action on $\rr^2$ preserves the families of vertical, horizontal, and diagonal lines, so edges in $L$ have a well-defined type.
    A nontrivial $4$-cycle $C$ through the basepoint in $L$ uses at most two of these directions.
    If $C$ uses one direction, then the vertices in that cycle are identified with a subgroup of $G$ of order $4$, and since $\abs{G}\geq 12$, translating gives us two disjoint cosets, so we get three disjoint cycles.

    Suppose that the cycle uses two directions $a$ and $b$.
    If there are two $a$ edges, and two $b$ edges, then $2a + 2b = 0$ in $G$, and hence $a+b=-a-b$.
    Since we can replace consecutive edges $ab$ with $ba$ and vice versa, we can assume that $C=abab$, and the set of vertices of $C$ is $\{0, a, a+b, 2a+b=-b \}$.
    Then the set $D$ of differences between vertices in the cycle is the $6$-element set $\{0, a, -a, b, -b, a+b\}$.
    Note that the third direction $a-b$ is not in $D$; as that difference would come from non-adjacent vertices of $C$, and that would contradict fullness of the $4$-cycle.
    Thus, shifting by the third direction $a-b$ gives a disjoint copy $a-b+C$.

    To find a third copy, we need to find $g \in G$ so that $(g+C) \cap C = \emptyset$ and $(g+C) \cap (a-b+C) = \emptyset$.
    This means that $g \notin D$ and $g \notin a-b+D$.
    Since shifting $D$ by $a-b$ takes $b$ to $a$ and $-a$ to $-b$, the set $D \cup (a-b+D)$ has at most $10$ elements, and since $\abs{G} \geq 12$, we can find such a $g$.

    If there are three $a$ edges and one $b$ edge, we can take $C=aaab$.
    Then $3a + b = 0$ in $G$, so $G$ is cyclic and generated by $a$, and the order of $a$ is $\geq 12$.
    The set of vertices of $C$ is then $\{0, a, 2a, 3a\}$, and shifting by $4a$ and $8a$ then gives three disjoint cycles.
\end{proof}

\subsection{\texorpdfstring{$\F$-$L^{2}$-Betti numbers}{F-L2-Betti numbers}} First, we need the following well-known observation.
\begin{lemma}\label{l:fields}
    Suppose $L$ is a flag complex, and $\F$ is a field.
    If $b_{1}^{(2)}(W_{L}; \F)=0$, then $b_{1}^{(2)}(W_{L}; \F')=0$ for any field $\F'$.
\end{lemma}
\begin{proof}
    By the universal coefficients we have $b^{(2)}_{1}(W_{L}; \Q) =0$.
    Then, Kielak's algebraic fibering theorem \cite{k20} implies $W_{L}$ virtually algebraically fibers, and hence by a version of L\"uck's mapping torus theorem \cite{aos24}*{Theorem 2.10} and multiplicativity, we get $b^{(2)}_{1}(W_{L}; \F') =0$ for any field $\F'$.
\end{proof}
\begin{proposition}\label{p:acyclic}
    $W_{L}$ is $\F$-$L^2$-acyclic if $L$ is a flag triangulation of:
    \begin{enumerate}
        \item A $2$-sphere.
        \item A $2$-disk with $4$-cycle boundary.
        \item A $4$-annulus, possibly pinched or edge-pinched.
    \end{enumerate}
\end{proposition}
\begin{proof}
    Since by Corollary~\ref{c:sphere} the only minimal flag triangulation of $S^{2}$ is the octahedron $O^{2}$, and since $W_{O^{2}}$ is $\F$-$L^2$-acyclic for all $\F$, Theorem~\ref{t:inequality} gives the first claim over $\F_{2}$.
    For other fields, the claim follows from Lemma~\ref{l:fields} and Poincar\'e duality, since $P_{L}$ is a closed $3$-manifold.

    If $L$ is a $4$-annulus, possibly pinched or edge-pinched, or a $2$-disk with full $4$-cycle boundary, then by Mayer--Vietoris coning off the boundaries does not change $b^{(2)}_{*}$ and results in a flag $2$-sphere, so the claim follows from the first one.
    If the boundary of the $2$-disk is not full, then the disk is a suspension of an edge, and the claim follows from the K\"unneth formula.
    cxc
\end{proof}
The first claim proves the $\F$-Singer Conjecture for $W_{L}$ for flag triangulations of $S^{2}$.
The claims were known before; as far as we know, the next computation is new.
\begin{proposition}\label{p:vanishing}
    We have $b^{(2)}_{1}(W_{L}; \F) = 0$ if $L$ is a flag triangulation of:
    \begin{enumerate}
        \item A $2$-sphere with a simplex.
        \item A $4$-annulus with a simplex.
    \end{enumerate}
\end{proposition}
\begin{proof}
    By Lemma~\ref{l:fields}, it is enough to work over $\F_{2}$.
    To prove the first claim, suppose $L=S^{2} \cup u*\gt$.
    Since contraction of an admissible edge not in $u*\gt$ produces a new flag $2$-sphere with the same simplex, we can assume that $L$ is minimal relative to $u*\gt$ by Theorem \ref{t:inequality}.
    Then by Lemma~\ref{l:s2+e}, $L-u$ is a disk with $4$-cycle boundary, hence $b_{1}^{(2)}(L-u; \F_{2})=0$, and the claim follows from the Mayer--Vietoris sequence of the decomposition $L=(L-u) \mathop{\cup}\limits_{\Lk u } \St u$.

    If $L$ is a $4$-annulus with a simplex, we can cone off each of the boundaries without changing $b_{*}^{(2)}$.
    The result is a flag 2-sphere with a simplex, so we are done by the above.
\end{proof}
\begin{lemma}\label{l:ss}
    Suppose $L$ is a flag triangulation of a torus which contains three disjoint nontrivial $4$-cycles.
    Then $b_1^{(2)}(W_L; \F) = 0$.
\end{lemma}
\begin{proof}
    The three cycles divide $T^2$ into three full $4$-annuli.
    By Proposition~\ref{p:acyclic}, each of these is $\F$-$L^{2}$-acyclic, and the same holds for their pairwise intersections.
    The triple intersection is an empty simplex, which corresponds to the trivial group.
    Therefore, the Mayer--Vietoris spectral sequence of the cover of $L$ by the three annuli for homology with coefficients in $D_{\F C_{L}}$ has an $E^1_{ij}$-sheet with zero entries except for $E^{1}_{2,0}$, so the result follows.
\end{proof}
\begin{corollary}\label{c:annulussimplex}
    Suppose $L$ is a minimal flag triangulation of a torus and $A \subset L$ is a full subcomplex.
    We have $b^{(2)}_{1}(W_{L}; \F) = 0$ if $A$ is one of the following:
    \begin{enumerate}
        \item A $4$-annulus.
        \item A pinched $4$-annulus.
        \item A $4$-annulus with a simplex.
    \end{enumerate}
\end{corollary}
\begin{proof}
    In every case, by Propositions~\ref{p:acyclic} and \ref{p:vanishing}, we have $b^{(2)}_{1}(W_A; \F) = 0$.
    This implies $b^{(2)}_{1}(W_{M(A)} \F) = 0$ and the claims follow if $M(A)=L$.
    So we will assume that $M(A) \neq L$.

    1. If $A$ is a full $4$-annulus, then $L$ contains a maximal full $4$-annulus $A'$.
    If $M(A')=L$, we are done, otherwise by Lemma \ref{l:full annulus} $L$ contains three disjoint nontrivial $4$-cycles and the claim follows from Lemma~\ref{l:ss}.

    2. If $A$ is a full pinched $4$-annulus, we can apply Lemma \ref{l:pinchedannulus} to conclude that $L$ contains a full $4$-annulus and then the claim follows from 1.

    3. If $A$ is a full $4$-annulus with a simplex, then by Lemma \ref{l:acyclicsubcomplex} $L$ contains either three disjoint nontrivial $4$-cycles or a full pinched $4$-annulus $A'$.
    In the former case, we are done by Lemma~\ref{l:ss}, and in the latter case by 2.
\end{proof}
\begin{corollary}\label{c:deg4}
    If $L$ is a minimal flag triangulation of a torus which contains a vertex of degree $4$, then $b^{(2)}_{1}(W_{L}; \F)=0$.
\end{corollary}
\begin{proof}
    Suppose $v$ is a vertex of degree $4$, and let $C_{1}$ and $C_{2}$ be the two diagonal full $4$-cycles crossing at $v$.
    Then the union of the cycles with the link of $v$ is either full or has one extra edge between the vertices opposite to $v$ in the cycles.
    In either case the connectedness assumptions of Lemma~\ref{l:crossing cycles} are satisfied, and since $b^{(2)}_{1}(W_{\Lk v}; \F)=0$, we are done by Corollary~\ref{c:annulussimplex}.
\end{proof}
\begin{lemma}\label{l:deg5}
    If $L$ is a minimal flag triangulation of a torus which contains a vertex of degree $5$, then $b^{(2)}_{1}(W_{L}; \F)=0$.
\end{lemma}
\begin{proof}
    Suppose $v$ is a vertex of degree $5$ with link $abdce$, and the two nontrivial $4$-cycles given by Lemma~\ref{l:vertex} are $C_{1}=vbyd$ and $C_{2}=vaxc$.
    Cutting the torus along the cycles gives the left picture in Figure~\ref{f:degree5.1}.
    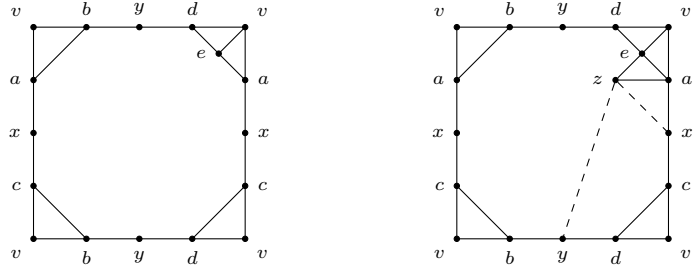
\begin{figure}[htb]
        \begin{center}
            \begin{tikzpicture}[scale=0.7]
                \begin{scope}[xshift=-4cm] \pic[scale=0.7]{deg5};
                \end{scope}
                \begin{scope}[xshift=4cm] \pic[scale=0.7]{deg5};
                    \node[name=z, vertex, label=left:$z$] at (1,1) {};
                    \draw (a2)--(z)--(e);
                    \draw[dashed] (x2)--(z)--(y1);
                \end{scope}
            \end{tikzpicture}
        \end{center}
        \caption{A degree $5$ vertex and two nontrivial $4$-cycles through it.
        We can assume that the only possible additional edges between the shown vertices are the dashed edges.
        }\label{f:degree5.1}
    \end{figure}

    We now show that we can reduce to the case where $C_{1}\cup C_{2} \cup \St v$ is full, in other words where the picture does not have additional edges between the shown vertices.
    Indeed, we can assume that no vertex has degree $4$, as otherwise we would be done by Corollary~\ref{c:deg4}.
    Then the intersection of links of two nonadjacent vertices does not contain adjacent edges, as that would give a degree $4$ vertex.
    In particular, if one of the vertices is $v$, then the intersection is a subcomplex of a complex consisting of a vertex and a disjoint edge.
    Therefore, $x$ is not adjacent to $b$, and $\Lk x \cap C_{1} \subset dy$ and hence connected.
    The intersections of the links of $a$ and $c$ with $C_{1}$ contain an edge, and therefore are also connected.
    Similarly, the intersections of the links of $b$, $y$, and $d$ with $C_{2}$ are also connected.

    If there exists a vertex $u \notin \St v$ such that $\Lk u \cap \Lk v$ is a vertex and a disjoint edge, then $A := \St v \cup u$ is an edge-pinched annulus and by Proposition~\ref{p:vanishing} we have $b^{(2)}_{1}(W_{A}; \F) = 0$.
    Since $A \subset M(\Lk v)$, $b^{(2)}_{1}(W_{M(\Lk v)}) = 0$, and by Lemma~\ref{l:crossing cycles} and Corollary~\ref{c:annulussimplex}, we are done.

    So we can assume no such vertex exists, and hence $\Lk x \cap \Lk v = \{a, c\}$ and $\Lk y \cap \Lk v = \{b, d\}$.
    An edge $xy$ would then either make a trivial $4$ or $5$-cycle without diagonals, which by Lemma~\ref{l:full cycle} must be filled with a cone.
    This creates either a new degree $4$ vertex, or makes $e$ degree $4$.
    Thus $C_{1} \cup C_{2} \cup \St v$ is full.

    Take $z \in \Lk (ae) - v$, so $z \notin C_{1} \cup C_{2} \cup \St v$.
    Then by our assumption $\Lk z \cap \Lk v = ae$, see Figure~\ref{f:degree5.1}.
    Note that $z$ cannot be connected to the left $x$ as that creates an empty triangle $axz$, or to the top $y$ as that creates a trivial $4$-cycle $dezy$ and hence a degree $4$ vertex.

    Therefore, a nontrivial $4$-cycle through $ez$ does not pass through $v$ and has one of the forms $ezxz'$ or $ezyd$ and does not intersect one of the outside $4$-cycles, so we get two annuli.
    Consider the annulus containing the triangle $aez$, see Figure~\ref{f:degree5.3}.
    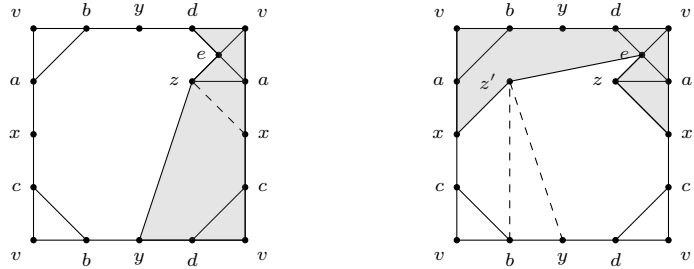
\begin{figure}[htb]
        \begin{center}
            \begin{tikzpicture}[scale=0.7]
                \begin{scope}[xshift=-4cm] \pic[scale=0.7]{deg5};
                    \coordinate (z) at (1,1);
                    \node[vertex, label=left:$z$] at (z) {};
                    \draw (a2)--(z)--(e);
                    \draw[dashed] (x2)--(z);
                    \draw[fill=gray, fill opacity=.2] (v1)--(d2)--(e)--(z)--(y1)--(v4)--cycle;
                \end{scope}
                \begin{scope}[xshift=4cm] \pic[scale=0.7]{deg5};
                    \coordinate (z) at (1,1);
                    \node[vertex, label=left:$z$] at (z) {};
                    \coordinate (z') at (-1,1);
                    \node[vertex, label=left:$z'$] at (z') {};
                    \draw (a2)--(z)--(e);
                    \draw (z)--(x2);
                    \draw[dashed] (b1)--(z')--(y1);
                    \draw[fill=gray, fill opacity =.2] (v2)--(x1)--(z')--(e)--(z)--(x2)--(v1);
                \end{scope}
            \end{tikzpicture}
        \end{center}
        \caption{The annulus on the left is full, the annulus on the right is full possibly with a simplex.}\label{f:degree5.3}
    \end{figure}

    For the cycle $ezyd$, this annulus is full, since the only possible extra edge $zx$ is inside the annulus, since $za$ is.
    For $ezxz'$, we can assume that $z$ is not adjacent to $y$ (otherwise use $ezyd$).
    Then the only outside edges between its boundary components are from $z'$.
    The edge $z'd$ is inside the annulus because of $ed$, so the possible edges are $z'b$ and $z'y$ and we get a full annulus with a simplex.
    In both cases we are done by Corollary~\ref{c:annulussimplex}.
\end{proof}
\begin{theorem}\label{t:torus}
    Suppose $L$ is a flag triangulation of a torus.
    Then $b^{(2)}_{\neq 2}(W_L; \F) = 0$ and $b^{(2)}_{2}(W_L; \F) = 1$.
\end{theorem}
\begin{proof}
    Since $P_{L}$ is a $3$-dimensional pseudo-manifold, $b^{(2)}_{> 2}(W_L; \F)=0$.
    Akita showed in \cite{a00} that if $L$ is a flag triangulation of a closed genus $g$ surface, then the orbihedral Euler characteristic of $W_L$ is equal to $g$.
    So by Lemma~\ref{l:fields} it is enough to show $b^{(2)}_{1}(W_L; \F_{2})=0$.
    By Theorem \ref{t:inequality} we can assume that $L$ is minimal.
    If $L$ has a vertex of degree $4$ or $5$, the claim is immediate from Corollary \ref{c:deg4} and Lemma~\ref{l:deg5}.
    Otherwise every vertex has degree 6, and the claim follows from Lemmas \ref{l:deg6} and \ref{l:ss}.
\end{proof}

In \cite{fmn94}, Fisk, Mohar, and Nedela showed that the real projective plane has only five minimal flag triangulations.
This makes computing the $L^2$-Betti number of $W_{\rr \mathrm{P}^2}$ much easier than the torus, because all five flag minimal triangulations have a $4$-cycle $C$ with $M(C) = L$, see Figure~\ref{f:RP2}.
Therefore,
\begin{theorem}\label{t:RP2}
    If $L$ is a flag triangulation of $\rr\mathrm{P}^{2}$, then $b^{(2)}_{2}(W_L ; \F) = 1/2$ and $b^{(2)}_{\neq 2}(W_L ; \F) = 0$.
\end{theorem}
\begin{figure}[htb]
    \begin{center}
        \begin{tikzpicture}
            \tikzmath{ \R=2.5; \s=1.7; \r=\R/\s; }
            \begin{scope}[xshift=1.75 cm]
                \node[regular polygon, regular polygon sides=10, minimum size=\R cm, draw, name=a, rotate=18] {};
                \node[star, star points=5, star point ratio=\s, minimum size=\R cm, draw, name=b] {};
                \node[regular polygon, regular polygon sides=5, minimum size=\r cm, rotate=36, name=c, draw] {};
                \foreach \i in {2,4,...,10}
                	   \draw (a.corner \i ) -- (a.center);
                \draw[very thick] (a.corner 4) -- (a.center) -- (a.corner 10) ;
            \end{scope}
            \begin{scope}[xshift=4.75 cm]
                \node[regular polygon, regular polygon sides=10, minimum size=\R cm, draw, name=a, rotate=18] {};
                \node[star, star points=5, star point ratio=\s, minimum size=\R cm, draw, name=b] {};
                \node[regular polygon, regular polygon sides=5, minimum size=\r cm, rotate=36, name=c, draw] {};
                \foreach \i in {2,4,8,10}
                     \draw (a.corner \i ) -- (a.center);
                \draw (a.corner 5) -- (a.corner 7);
                \draw (c.center) --(c.corner 3) ;
                \draw[very thick] (a.corner 4) -- (a.center) -- (a.corner 10) ;
            \end{scope}
            \begin{scope}[xshift=0 cm, yshift=-3 cm]
                \node[regular polygon, regular polygon sides=12, minimum size=\R cm, draw, name=a, rotate=15] {};
                \node[star, star points=3, star point ratio=\s, minimum size=\R cm, draw, name=b] {};
                \node[star, star points=3, star point ratio=\s, minimum size=\r cm, rotate=60, name=c, draw] {};
                \foreach \i in {1,3,...,12}
                     \draw (a.corner \i ) -- (a.center);
                \foreach \i in {1,2,3} \foreach \j in {1,2,3} { 
                    \pgfmathtruncatemacro{\k}{mod(\i * 4-\j, 12)+1}
               	    \draw (c.outer point \i ) -- (a.corner \k ); 
	        } 	    
                \draw[very thick] (a.corner 5) -- (a.corner 11) ;
            \end{scope}

            \tikzmath{ \sa=\s*2.2; \r=\R/\s; \ra=\R/\sa; }
            \begin{scope}[xshift=3 cm, yshift=-3 cm]
                \node[regular polygon, regular polygon sides=12, minimum size=\R cm, draw, name=a, rotate=15] {};
                \node[star, star points=3, star point ratio=\s, minimum size=\R cm, draw, name=b] {};
                \node[regular polygon, regular polygon sides=3, minimum size=\ra cm, rotate=60, name=c, draw] {};
                \node[star, star points=3, star point ratio=\sa, minimum size=\R cm, draw, name=d] {};
                \foreach \i in {1,2,3} {
                    \foreach \j in {1,2,3} { 
                    	\pgfmathtruncatemacro{\k}{mod(\i * 4-\j, 12)+1}
                     	\draw (b.inner point \i ) -- (a.corner \k ); 
	             }
                     \draw (b.inner point \i ) -- (c.corner \i ); 
                }
                \draw[very thick] (a.corner 5) -- (c.corner 1) -- (c.corner 3) -- (a.corner 11) ;
            \end{scope}
            \tikzmath{ \sa=\s*3; \ra=\R/\sa; \ys= 0.66; \y= (1-\ys)*.45 ; }
            \begin{scope}[xshift=6 cm, yshift= -3 cm]
                \node[regular polygon, regular polygon sides=12, minimum size=\R cm, draw, name=a, rotate=15] {};
                \node[star, star points=3, star point ratio=\s, minimum size=\R cm, draw, name=b] {};
                \node[regular polygon, regular polygon sides=3, minimum size=\ra cm, rotate=60, name=c] {};
                \node[star, star points=3, star point ratio=1.4, minimum size=\r cm, draw, yscale=\ys, rotate=60, name=d] at (0,\y cm) {};
                \foreach \i in {1,2,3} \foreach \j in {1,2,3} {
                    \pgfmathtruncatemacro{\k}{mod(\i * 4-\j, 12)+1}
                    \draw (b.inner point \i ) -- (a.corner \k ); 
                }
                \foreach \i in {1,2}{
                    \draw (d.inner point \i ) -- (d.inner point 3); \pgfmathtruncatemacro{\k}{mod(\i * 4, 12)+1}
                    \draw (d.inner point \i ) -- (a.corner \k);
                    \draw (d.outer point 2) -- (a.corner \k); 
                }
                \draw (b.inner point 2) -- (a.corner 1);
                \draw[very thick] (a.corner 5) -- (d.inner point 1) -- (d.inner point 3) --(b.inner point 3)-- (a.corner 11) ;
            \end{scope}
        \end{tikzpicture}
    \end{center}
    \caption{The five minimal flag triangulations of $\rr \mathrm{P}^2$, with the highlighted $4$-cycle $C$ with $M(C) = L$.}\label{f:RP2}
\end{figure}
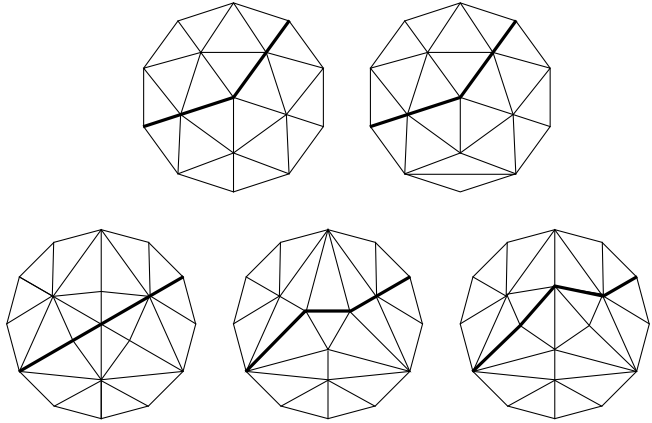

\section{Higher dimensions}

Let $\mathcal{M}$ be the family of minimal PL triangulations of even-dimensional spheres.
\begin{theorem}
    The $\F_2$-Singer Conjecture holds for all $W_L$ for PL spheres $L$ if and only if it holds for all $W_L$ with $L \in \mathcal{M}$.
\end{theorem}
\begin{proof}
    We first assume $L$ is a flag piecewise linear triangulation of $S^{2n}$.
    By induction on $n$, we can assume that the $\F_2$-Singer Conjecture holds for all $W_L$ if $L$ is a piecewise linear triangulation of $S^{2n-2}$.
    Therefore, by Theorem \ref{t:inequality}, contracting an admissible edge in $L$ can only raise $b^{(2)}_{*}(W_{L}; \F_{2})$.

    To get from flag triangulations of $S^{2n}$ to flag triangulations of $S^{2n+1}$, we use the same Mayer--Vietoris arguments from \cite{do01} and \cite{aos21}.
    In particular, iteratively applying \cite{aos21}*{Lemma 12} shows that a counterexample to the $\F_2$-Singer Conjecture for $L$ a flag triangulation of $S^{2n+1}$ would imply a counterexample to the $\F_2$-Singer Conjecture for the link of a simplex in $L$.
\end{proof}

In \cite{aos25} we proved that for (the inverses of) edge subdivisions, the inequality of Theorem~\ref{t:inequality} holds over any field.
In the case where $L$ is an even-dimensional sphere, the $\F$-version of the Singer Conjecture would predict that the above inequality is an equality.
However, given an odd $p$ and odd $n > 7$, we constructed in \cite{aos21} flag PL triangulations $L$ of $S^{n-1}$ so that $P_L$ contradicted the $\Fp$-Singer Conjecture.
By Lutz and Nevo's result~\cite{ln16}, these triangulations can be connected to the $n$-octahedron by a sequence of edge subdivisions and their inverses.
Therefore, at some point in this sequence, $\F_p$-$L^2$-Betti numbers must change, and hence in odd characteristic, we can have a strict inequality even when $L$ is a flag PL triangulation of $S^{n-1}$.
Corollary \ref{t:inequality} fails for other $\F_p$-$L^2$-Betti numbers, even when $L$ is a flag PL triangulation of $S^{n-1}$.
It is still possible that the equality holds over $\F_{2}$ or $\Q$ for spheres, but the following example shows that in general, subdividing edges with $\F$-$L^2$-acyclic links can affect the $\F$-$L^2$-Betti numbers for any field $\F$.
\begin{example}
    Let $L$ be the flag complex obtained from $K_{3,3}$ by suspending an edge $e$.
    The link of $e$ in $L$ is two suspension points, so $W_{\Lk e}$ is $\F$-$L^2$-acyclic.
    Then easy Mayer--Vietoris arguments show that $b^{(2)}_{1}(W_{L}; \F)=b^{(2)}_2(W_{L}; \F)=1/4$, but subdividing $e$ results in a $\F$-$L^2$-acyclic group.
\end{example}
\begin{remark}\label{r:1}
    If $n > 2$, then there are infinitely many minimal flag triangulations of $S^{n}$.
    The easiest way to construct these is by doubling; if $L$ is such a triangulation and $v \in L$, then $D_v(L) := (L - v) \mathop{\cup}\limits_{\Lk v} (L - v)$ is another minimal flag triangulation.
    If $v$ is not a suspension point, $L \neq \{v, v'\} * L'$, then $D_v(L)$ is distinct from $L$.
    For instance, the minimal flag triangulation of $S^3$ constructed in \cite{v22} gives infinitely many minimal flag $S^3$'s, and taking suspensions yields higher-dimensional examples.
    These examples are not interesting from the perspective of $\F$-$L^2$-Betti numbers, since $W_{D_v(L)}$ is an index-two subgroup of $W_L$.

    Karim Adiprasito pointed out to us that the unsolvability of the PL sphere recognition problem yields infinitely many minimal flag triangulations of $S^n$ for $n \geq 5$, and these flag triangulations do not come from iteratively doubling from a finite number of seed minimal triangulations.
    The unsolvability of the problem arises from the fact that no algorithm can decide which of the homology spheres have trivial $\pi_1$.
    If all minimal flag triangulations were iterated doubles of a finite collection of minimal flag spheres, then one could solve the recognition problem by first taking the barycentric subdivision of a homology sphere $S$ so that it becomes flag, then contracting it to a minimal flag $M$, and then finally constructing algorithmically all doubles with the same number of vertices as $M$ and checking for each double if it is $M$.

    Note that one can construct all minimal flag PL triangulations of $S^{n}$ algorithmically, for example by starting with an $n$-octahedron, performing all possible sequences of edge subdivisions and their inverses in the order of increasing length and contracting each resulting sphere to a minimal one.
    By Lutz--Nevo this algorithm will eventually produce all minimal flag PL spheres.
    However, in any such algorithm, one would not know in a finite time if all minimal spheres of a given size were constructed, since otherwise one could solve the recognition problem as above.
    In other words, the number of minimal flag PL $n$-spheres as a function of the number of vertices is not computable.
\end{remark}

\noindent Grigori Avramidi, Mathematical Institute, Andrew Wiles Building, Observatory Quarter, University of Oxford, Oxford, United Kingdom\\
\url{grigori.avramidi@maths.ox.ac.uk}

\noindent Boris Okun, University of Wisconsin-Milwaukee, Department of Mathematical Sciences, PO Box 413, Milwaukee, WI 53201-0413\\
 \url{okun@uwm.edu}

\noindent Kevin Schreve, Louisiana State University, Department of Mathematics, Baton Rouge, LA, 70806\\
 \url{kschreve@lsu.edu}


\begin{bibdiv}
    \begin{biblist}

 \bib{a08}{article}{
    author = {Agol, Ian},
    title = {Criteria for virtual fibering},
    date = {2008},
    issn = {1753-8416},
    journal = {J. Topol.},
    volume = {1},
    number = {2},
    pages = {269\ndash 284},
    url = {https://mathscinet.ams.org/mathscinet-getitem?mr=2399130},
    review = {\MR{2399130}}, }

 \bib{a00}{article}{
    author = {Akita, Toshiyuki},
    title = {Euler characteristics of Coxeter groups, PL-triangulations of closed manifolds, and cohomology of subgroups of Artin groups},
    date = {2000},
    issn = {1753-8416},
    journal = {Journal of the LMS},
    volume = {61},
    number = {2},
    pages = {721\ndash 736},
}

\bib{a70}{article}{
    author = {Andreev, E.~M.},
    title = {On convex polyhedra in {L}oba\v{c}evski\u{i} spaces},
    date = {1970},
    journal = {Math.USSR Sb.},
    volume = {10},
    pages = {413\ndash 440}, 
}

\bib{a70a}{article}{
    author = {Andreev, E.~M.},
    title = {On convex polyhedra of finite volume in {L}oba\v{c}evski\u{i} space},
    date = {1970},
    journal = {Math.USSR Sb.},
    volume = {12},
    pages = {255\ndash 259}, }

\bib{aos21}{article}{
    author = {Avramidi, Grigori},
    author = {Okun, Boris},
    author = {Schreve, Kevin},
    title = {Mod {$p$} and torsion homology growth in nonpositive curvature},
    date = {2021},
    issn = {0020-9910},
    journal = {Invent. Math.},
    volume = {226},
    number = {3},
    pages = {711\ndash 723},
    url = {https://mathscinet.ams.org/mathscinet-getitem?mr=4337971},
    review = {\MR{4337971}}, }

 \bib{aos24}{article}{
    author = {Avramidi, Grigori},
    author = {Okun, Boris},
    author = {Schreve, Kevin},
    title = {Homology growth, hyperbolization, and fibering},
    date = {2024},
    journal = {GAFA},
    volume = {34},
    pages = {303\ndash376}, }

 \bib{aos25}{article}{
    author = {Avramidi, Grigori},
    author = {Okun, Boris},
    author = {Schreve, Kevin},
    title = {Edge subdivisions and the $L^2$-homology of right-angled Coxeter groups},
    volume = {114},
    number = {1},
    journal = {Journal of the LMS},
    date = {2026}, }

\bib{blls14}{article}{
    author = {Bergeron, Nicolas},
    author = {Linnell, Peter},
    author = {L{\"u}ck, Wolfgang},
    author = {Sauer, Roman},
    title = {On the growth of Betti numbers in $p$-adic analytic towers},
    date = {2014},
    journal = {Groups, Geometry and Dynamics},
    volume = {8},
    pages = {311\ndash 329},

}

\bib{bowwzz20}{article}{
    author = {Bibby, Christin},
    author = {Odesky, Andrew},
    author = {Wang, Mengmeng},
    author = {Wang, Shuyang},
    author = {Zhang, Ziyi},
    author = {Zheng, Hailun},
    title = {Minimal flag triangulations of lower-dimensional manifolds},
    journal = {Involve},
    volume = {13},
    date = {2020},
    number = {4},
    pages = {683--703},
    issn = {1944-4176},
    review = {\MR{4190432}}, doi={10.2140/involve.2020.13.683}, }

\bib{dbook}{book}{
    author = {Davis, Michael W.},
    title = {The geometry and topology of Coxeter groups},
    series = {London Mathematical Society Monographs Series},
    volume = {32}, publisher={Princeton University Press, Princeton, NJ},
    date = {2008},
    pages = {xvi+584} }

 \bib{do01}{article}{
    author = {Davis, Michael~W.},
    author = {Okun, Boris},
    title = {Vanishing theorems and conjectures for the {$\ell^2$}-homology of right-angled {C}oxeter groups},
    date = {2001},
    issn = {1465-3060},
    journal = {Geom. Topol.},
    volume = {5},
    pages = {7\ndash 74},
    url = {http://dx.doi.org/10.2140/gt.2001.5.7},
    review = {\MR{1812434 (2002e:58039)}}, }

 \bib{d79}{article}{
    author = {Dodziuk, Jozef},
    title = {$L^2$-harmonic forms on rotationally symmetric Riemannian manifolds},
    date = {1979},
    journal = {Proceedings of the AMS},
    volume = {77},
    number = {3},
    pages = {395\ndash 401}, }

 \bib{fmn94}{article}{
    author = {Fisk, Steve},
    author = {Mohar, Bojan},
    author = {Nedela, Roman},
    title = {Minimal Locally Cyclic Triangulations of the Projective Plane},
    date = {1994},
    journal = {Journal of Graph Theory},
    volume = {18},
    pages = {25\ndash 35}, }

 \bib{hr91}{article}{
    author = {Hartsfield, Nora},
    author = {Ringel, Gerhard},
    title = {Clean triangulations},
    date = {1991},
    issn = {0025-5521},
    journal = {Combinatorica},
    volume = {11},
    pages = {145\ndash 155}, }

 \bib{k20}{article}{
    author = {Kielak, Dawid},
    title = {Residually finite rationally solvable groups and virtual fibring},
    date = {2020},
    issn = {0894-0347},
    journal = {J. Amer. Math. Soc.},
    volume = {33},
    number = {2},
    pages = {451\ndash 486},
    url = {https://doi.org/10.1090/jams/936},
    review = {\MR{4073866}}, }

\bib{l94a}{article}{
    author = {L{\"u}ck, Wolfgang},
    title = {${L}\sp 2$-{B}etti numbers of mapping tori and groups},
    date = {1994},
    issn = {0040-9383},
    journal = {Topology},
    volume = {33},
    number = {2},
    pages = {203\ndash 214},
    review = {\MR{95g:58235}}, }

 \bib{ln16}{article}{
    author = {Lutz, Frank~H.},
    author = {Nevo, Eran},
    title = {Stellar theory for flag complexes},
    date = {2016},
    issn = {0025-5521},
    journal = {Math. Scand.},
    volume = {118},
    number = {1},
    pages = {70\ndash 82},
    url = {https://doi.org/10.7146/math.scand.a-23297},
    review = {\MR{3475102}}, }

\bib{mm92}{article}{
    author = {Malni{\v c}, Aleksander},
    author = {Mohar, Bojan},
    title = {Generating locally cyclic triangulations of surfaces},
    date = {1992-11},
    issn = {0095-8956},
    journal = {Journal of Combinatorial Theory, Series B},
    volume = {56},
    number = {2},
    pages = {147\ndash 164},
    url = {http://dx.doi.org/10.1016/0095-8956(92)90015-P}, }

 \bib{mn95}{article}{
    author = {Malni\v{c}, A.},
    author = {Nedela, R.},
    title = {{$k$}-minimal triangulations of surfaces},
    date = {1995},
    issn = {0862-9544,1336-0310},
    journal = {Acta Math. Univ. Comenian. (N.S.)},
    volume = {64},
    number = {1},
    pages = {57\ndash 76},
    review = {\MR{1360987}}, }

\bib{m19}{thesis}{
    author = {Melzer, Sebastian},
    title = {$k$-irreducible triangulations of $2$-manifolds}, type={Ph.D. Thesis},
    date = {2019}, 
    eprint={https://nbn-resolving.org/urn:nbn:de:bsz:14-qucosa2-356439} }

 \bib{v22}{article}{
    author = {Venturello, Lorenzo},
    title = {On flag spheres with few equators},
    date = {2022}, eprint={https://arxiv.org/pdf/2203.10003.pdf}, }

 \bib{v15}{article}{
    author = {Volodin, V.~D.},
    title = {Combinatorics of flag simplicial 3-polytopes},
    date = {2015},
    issn = {0042-1316},
    journal = {Uspekhi Mat. Nauk},
    volume = {70},
    number = {1(421)},
    pages = {181\ndash 182},
    url = {https://doi.org/10.4213/rm9645},
    review = {\MR{3353120}}, }

\end{biblist}
\end{bibdiv}
\end{document}